\documentclass[amssymb,verbatim,10pt]{amsart}
\usepackage[fontsize = 10bp]{fontsize}

\usepackage{extarrows}

\usepackage{amssymb,latexsym, amsmath, amscd, array, graphicx,pb-diagram}
\usepackage{hyperref}
\usepackage{color}

\usepackage{mathtools}
\usepackage{stackengine}
\setstackEOL{\\}
\usepackage{tikz-cd}

\usetikzlibrary{decorations.markings}
\tikzset{negated/.style={
        decoration={markings,
            mark= at position 0.75 with {
                \node[transform shape] (tempnode) {$\backslash$};
            }
        },
        postaction={decorate}
    }
}

\makeatletter
\def\@tocline#1#2#3#4#5#6#7{\relax
  \ifnum #1>\c@tocdepth 
  \else
    \par \addpenalty\@secpenalty\addvspace{#2}%
    \begingroup \hyphenpenalty\@M
    \@ifempty{#4}{%
      \@tempdima\csname r@tocindent\number#1\endcsname\relax
    }{%
      \@tempdima#4\relax
    }%
    \parindent\z@ \leftskip#3\relax \advance\leftskip\@tempdima\relax
    \rightskip\@pnumwidth plus4em \parfillskip-\@pnumwidth
    #5\leavevmode\hskip-\@tempdima
      \ifcase #1
       \or\or \hskip 1em \or \hskip 2em \else \hskip 3em \fi%
      #6\nobreak\relax
    \hfill\hbox to\@pnumwidth{\@tocpagenum{#7}}\par
    \nobreak
    \endgroup
  \fi}
\makeatother

\usepackage{color}
\hypersetup{
    colorlinks=true,
    linkcolor=red,
    urlcolor=black,
    citecolor=blue
           }
           
\usepackage{amsthm}
\usepackage{hyperref}
\usepackage[all]{xy}
\usepackage[scr]{rsfso}

\swapnumbers 

\theoremstyle{plain}

\newtheorem{thm}{Theorem}[section]
\newtheorem{theorem}[thm]{Theorem}

\newtheorem{lemma}[thm]{Lemma}
\newtheorem{conjec}[thm]{Conjecture}
\newtheorem{prop}[thm]{Proposition}
\newtheorem{cor}[thm]{Corollary}

\newcommand\OO{{\mathcal O}}

\theoremstyle{definition}

\newtheorem{definition}[thm]{Definition}

\newtheorem{remark}[thm]{Remark}

\newtheorem{ex}[thm]{Example}

\newtheorem{question}[thm]{Question}

\def\r{\protect\operatorname{rank}}

\DeclareMathOperator{\cd}{{\rm cd}}

\DeclareMathOperator{\M}{{\mathcal{M}}}

\usepackage{enumerate}

\def\pr{\textup{proj}}

\def\id{\protect\operatorname{id}}
\def\Im{\protect\operatorname{Im}}

\def\ch{\protect\operatorname{ch}}

\def\pr{\protect\operatorname{pr}}

\def\scr{\mathcal}
\def\A{{\scr A}}
\def\B{{\scr B}}

\def\C{{\mathbb C}}
\def\Z{{\mathbb Z}}

\def\R{{\mathbb R}}

\usepackage[font=small]{caption}
\def\1{\hbox{\rm\rlap {1}\hskip.03in{\rom I}}}
\def\Bbbone{{\rm1\mathchoice{\kern-0.25em}{\kern-0.25em}
{\kern-0.2em}{\kern-0.2em}I}}

\def\wt{\widetilde}

\def\ov{\overline}

\usepackage{comment}

\long\def\forget#1\forgotten{}

\newcommand\ver[1]{\marginpar{\tiny Changed in Ver \VER}}

\date{\today}

\begin{document}

\title[Symplectically aspherical K\"ahler manifolds]{Symplectically aspherical K\"ahler manifolds, scalar curvature, and the fundamental group}

\author[L.~F.~Di~Cerbo]{Luca~F.~Di~Cerbo}

\author[A.~Dranishnikov]{Alexander~Dranishnikov} 

\author[E.~Jauhari]{Ekansh~Jauhari}

\address{Luca F. Di Cerbo, Department of Mathematics, University of Florida, 358 Little Hall, Gainesville, FL 32611-8105, USA.} 
\email{ldicerbo@ufl.edu}

\address{Alexander N. Dranishnikov, Department of Mathematics, University of Florida, 358 Little Hall, Gainesville, FL 32611-8105, USA.}
\email{dranish@ufl.edu}

\address{Ekansh Jauhari, Department of Mathematics, University of Florida, 358 Little Hall, Gainesville, FL 32611-8105, USA.}
\email{ekanshjauhari@ufl.edu}

\begin{abstract}
We present a detailed study of closed smooth manifolds having K\"ahler forms that pullback to exact forms on the universal cover. We show that these manifolds, which we call symplectically aspherical K\"ahler manifolds, exist in abundance, even outside the aspherical setting, and have interesting topological and geometric features, such as large fundamental group \'a la Koll\'ar and the absence of K\"ahler metrics of positive scalar curvature. Motivated by the latter, we extend the Gromov--Lawson Conjecture on aspherical manifolds to symplectically aspherical manifolds and prove it in the spin case. 

We also study K\"ahler cones on symplectically aspherical K\"ahler manifolds and the realizability problem of their fundamental group, and explore their other complex geometric properties.
\end{abstract}



\maketitle
\tableofcontents

\section{Introduction}
The study of closed K\"ahler manifolds having extra structures is ubiquitous in complex geometry. For instance, there are Gromov's \emph{K\"ahler hyperbolic manifolds} that are defined as those having a K\"ahler form whose pullback to the universal cover is $d$-bounded (and include, for example, K\"ahler manifolds having negative sectional curvature)~\cite{Gr-kahler hyperbolic}, and then there are Koll\'ar's manifolds having \emph{large fundamental group} (also called \emph{Koll\'ar large manifolds}) that are defined as those whose universal cover does not contain any positive dimensional compact complex subvarieties (and include, for example, K\"ahler manifolds having Stein universal cover)~\cite{Kollar1,Kollar2}. 

In this paper, we study \emph{symplectically aspherical K\"ahler manifolds} that are defined as those having a K\"ahler form whose pullback to the universal cover is exact, or equivalently, that vanishes on every real spherical homology class (such forms are called \emph{aspherical} in the literature). 
The class of these manifolds properly contains the class of K\"ahler hyperbolic manifolds and is contained within the class of Koll\'ar large manifolds, and manifolds in this class have tame topological properties that help obtain results of complex and Riemannian geometric flavors.

There are two primary motivations behind studying these K\"ahler manifolds. The first motivation comes from symplectic topology and geometry, where imposing the asphericity condition (first formulated by Floer in his study of the Arnol'd conjecture~\cite{Floer}) on \emph{symplectic} forms is fruitful for various purposes. Indeed, the Arnol'd conjecture holds true for \emph{symplectically aspherical manifolds}~\cite{RO}, and classical theories such as Floer homology, the Lusternik--Schnirelmann theory for Lagrangian intersections, and the theory of $J$-holomorphic curves are better understood for them (see~\cite{Sch,KRT2} and the references therein). Moreover, the algebraic topology of symplectically aspherical manifolds is well-studied~\cite{LO}. In view of these advantages, it is natural to ask what this topological condition of symplectic asphericity earns one in complex geometric contexts! 
Specifically, upon imposing symplectic asphericity on \emph{K\"ahler} forms, one can investigate consequences at the interface of complex and differential geometry, and approach them using techniques coming also from symplectic and algebraic topology.

The second motivation comes from the placement of symplectically aspherical K\"ahler manifolds in complex algebraic geometry. In the special setting of K\"ahler hyperbolic manifolds, all but possibly the middle $L^2$-Betti numbers of the universal cover vanish (so the Hopf--Singer conjecture is verified)~\cite{Gr-kahler hyperbolic} and the canonical line bundle is ample~\cite{CY}. In the generalized setting of Koll\'ar large manifolds, the K\"ahler universal cover is predicted by the Shafarevich conjecture to be a Stein manifold~\cite{Kollar1,Kollar2}. So, one can view the class of symplectically aspherical K\"ahler manifolds as an intermediate class to extend such results as well as to test such conjectures. As shown recently in~\cite{DCDJ}, several members of this class, remarkably many symmetric products of curves, show promise; see also interesting results obtained and directions pursued in~\cite{JZ,CX,CY2,BDET,DCL1,DCL2}. Therefore, it becomes quite natural to investigate various complex and differential geometric properties of general manifolds belonging to this class.


The goal of this paper is to answer several nuanced questions arising in the study of the geometry and topology of symplectically aspherical K\"ahler manifolds, and to pose some new ones. Of our particular interest in this paper are \emph{non-aspherical} symplectically aspherical K\"ahler manifolds, primarily those with non-trivial second homotopy group (see~\cite{Go2} and Section~\ref{sec: prelim} for motivations), which are especially nice in complex dimension $2$ and whose construction we describe in each dimension. We study which finitely presented discrete groups can be realized as the fundamental group of these manifolds, and give a complete answer in the case of free abelian groups. Our findings are analogous to those obtained in~\cite{IKRT,KRT1} for symplectically aspherical manifolds. We also investigate the positivity of scalar curvature associated with both Riemannian and K\"ahler metrics on these manifolds, and (with and without additional assumptions on the topology of the underlying manifold) we obtain \emph{non-existence results} that support some natural conjectures and extend related results from~\cite{DCDJ}. Of course, these results apply to the class of K\"ahler hyperbolic manifolds, and some of them also extend to that of Koll\'ar large manifolds. Then, we look at the question of what happens to the K\"ahler cone associated with aspherical K\"ahler forms. In this direction, we show that the asphericity of the entire K\"ahler cone is a very constraining condition, and we describe situations where the K\"ahler cone has no other aspherical K\"ahler forms. These results show, in particular, that closely related K\"ahler forms may behave quite differently when pulled back to the universal cover.

\subsection*{Organization of this paper} To make this paper somewhat self-contained, we begin with a review of symplectically aspherical manifolds in Section~\ref{sec: prelim}, where we also define symplectically aspherical K\"ahler manifolds and discuss examples. 

In Section~\ref{sec: in complex geom}, we give a general recipe for constructing \emph{non-aspherical} symplectically aspherical K\"ahler (SAK) manifolds in each complex dimension $\ge$ 2, and prove a comparison theorem, see Theorem~\ref{prop: other kahler}. There we compare the class of SAK manifolds with that of other well-studied K\"ahler manifolds, such as real hyperbolic, those having Stein universal cover, Kobayashi hyperbolic, K\"ahler hyperbolic, and those having large fundamental group (called Koll\'ar large here).

In Section~\ref{sec: fund grp}, we prove the following.

\begin{theorem}[Theorem~\ref{th: main1}]
Each free abelian group of even rank $\ge 4$ is realized as the fundamental group of a closed symplectically aspherical smooth projective surface of general type that has non-trivial real spherical homology classes and a Stein universal cover.    
\end{theorem}

In Section~\ref{sec: psc}, we present results on the scalar curvature associated with K\"ahler and Riemannian metrics on symplectically aspherical manifolds. In particular, we propose a symplectic version of the famous Gromov--Lawson conjecture.
\begin{conjec}
A closed symplectically aspherical manifold cannot support a Riemannian metric of positive scalar curvature.     
\end{conjec}

We prove this conjecture in dimensions $\le 4$ using the Seiberg--Witten theory (see Theorem \ref{th: 2 and 4}), and for all manifolds with spin universal covers in the case of abelian fundamental groups (see Corollary~\ref{cor: abelian fund grp}). For general fundamental groups and dimensions, we show the following.

\begin{theorem}[Theorem~\ref{th: c-symplectic psc}]\label{th: same}
A closed symplectically aspherical manifold having a spin finite cover cannot support a Riemannian metric of positive scalar curvature. 
\end{theorem}

The obstruction to metrics of positive scalar curvature obtained in Theorem~\ref{th: c-symplectic psc} also follows from Theorem~A in the recent work of Bei--Cecchini~\cite{BC}, whose approach is different from ours.

We then prove a strengthening of our conjecture for the scalar curvature associated with a K\"ahler metric in each complex dimension $\ge 2$. 

\begin{theorem}[Theorem~\ref{th: asphericalK}]
A smooth projective variety admitting a birational morphism onto a symplectically aspherical smooth projective variety cannot support a K\"ahler metric of positive scalar curvature.     
\end{theorem}

In Sections~\ref{sec: all aspherical} and~\ref{sec: non-aspherical}, we present a set of results concerning the study of the K\"ahler cones of symplectically aspherical K\"ahler manifolds. In particular, we look at the following two extremes: 
\begin{enumerate}
    \item when \emph{all} K\"ahler forms in the K\"ahler cone are aspherical;
    \item when there is an \emph{essentially unique} aspherical K\"ahler form. 
\end{enumerate} 
In the first situation, our results suggest that the topology of such manifolds is constrained, especially for K\"ahler surfaces.

\begin{theorem}[Corollary~\ref{cor: dim 4 all aspherical}]
Let $M$ be a closed K\"ahler surface with an integrable almost complex structure $J$. If every $J$-compatible K\"ahler form on $M$ is aspherical, then real spherical homology classes of $M$ cannot be represented by any symplectic $2$-sphere embedded with respect to any symplectic structure on $M$.  
\end{theorem}

For the second situation, we give an explicit family of examples.

\begin{theorem}[Theorem~\ref{th: unique Kaehler}]
Let a complex curve $M_g$ be generic in its moduli space. Its $n$-th symmetric product $SP^n(M_g)$ for $n\le\lfloor\tfrac{g+1}{2}\rfloor$ is a closed smooth projective variety that parametrizes effective divisors of degree $n$ on $M_g$, and the pullback $\mu_n^*\omega_0$ under the Abel--Jacobi map $\mu_n\colon SP^n(M_g)\to T^{2g}:=\C^g/\Z^{2g}$, where $\omega_0$ is the canonical aspherical K\"ahler form on $T^{2g}$, is an aspherical K\"ahler form on $SP^n(M_g)$. Let $J$ denote the corresponding almost complex structure on $SP^n(M_g)$.

Then a $J$-compatible K\"ahler form $\omega'$ on $SP^n(M_g)$ is aspherical if and only if it is cohomologous to a non-zero scalar multiple of $\omega$.    
\end{theorem}

Showing the latter involves the use of the description of the N\'eron--Severi group of these manifolds, along with cohomological computations \'a la Macdonald, see Section~\ref{sec: non-aspherical}.
\\

Throughout the paper, we mention several questions of different flavors (along with their motivations) that, we hope, will stimulate further interest and research.

\section{Definitions, examples, and overview} \label{sec: prelim}
Let $M$ be a closed smooth $2n$-manifold. We say that a closed $2$-form $\omega$ on $M$ is \emph{$c$-symplectic} ($c$ for ``cohomologically'') if $[\omega]^n\ne 0$ in $H^{2n}(M;\R)$. Of course, a symplectic form on $M$ is $c$-symplectic since its $n$-fold exterior product is the volume form on $M$, but the converse is not true in general (say, on $\C P^2\#\C P^2$). 

A $c$-symplectic form $\omega$ is called \emph{aspherical} if for every smooth map $f\colon S^2\to M$, we have 
\[
\int_Sf^*\omega=0.
\]
Equivalently, $\omega$ is aspherical if and only if $\langle[\omega],\Pi(M)\rangle=0$, where $\Pi(M)$ is the vector space of all real spherical homology classes of $M$, i.e., 
\[
\Pi(M)=\Im(h\otimes\id_\R\colon\pi_2(M)\otimes\R\to H_2(M;\R)),
\]
where $h\colon\pi_2(M)\to H_2(M;\Z)$ is the Hurewicz map.
If $M$ admits such a form, it is called \emph{$c$-symplectically aspherical}. A closed manifold equipped with an aspherical symplectic form is called \emph{symplectically aspherical}. 

We now consider manifolds equipped with aspherical K\"ahler forms.

\begin{definition}
    Let $(M,J)$ be a closed K\"ahler manifold. If there exists a K\"ahler form $\omega$ on $M$ compatible with $J$ that is aspherical in the aforementioned sense, then the triad $(M,J,\omega)$ is called a \emph{symplectically aspherical K\"ahler} manifold.
\end{definition}

If $\Pi(M)=0$ (for example, if $\pi_2(M)=0$), then each $c$-symplectic (and hence, K\"ahler) form on $M$ is aspherical. Conversely, it is easy to see that $\Pi(M)=0$ if every symplectic form on $M$ is aspherical (see, for instance,~\cite[Proposition~8.1]{IKRT}). We call such manifolds \emph{totally symplectically aspherical}. In Section~\ref{sec: all aspherical}, we will study \emph{totally symplectically aspherical K\"ahler} or \emph{totally K\"ahler aspherical} manifolds, i.e., those on which each K\"ahler form, compatible with a given integrable almost complex structure, is aspherical.

\begin{ex}\label{ex: sa and sak} 
\begin{enumerate}[(1)]
    \item For each $n$, the projective space $\C P^n$ is not symplectically aspherical since $\Pi(\C P^n)\cong H_2(\C P^n;\R)\cong\R$ by the Hurewicz theorem.
    \item  The Kodaira--Thurston manifold $\mathcal{KT}=\mathcal{H}_3\times S^1$ (where $\mathcal{H}_3$ is the Heisenberg $3$-manifold) and some of its branched covers are symplectically aspherical (see~\cite{Go2}), but never K\"ahler (see also~\cite{Thurston}).

    \item For each $n\ge 2$, the $n$-th symmetric product $SP^n(M_g)$ of a projective curve $M_g$ for $g\ge 2n-1$, where $M_g$ has a \emph{generic} complex structure in its moduli space $\M_g$, is a symplectically aspherical K\"ahler manifold, see~\cite[Theorem~4.1]{DCDJ} (and also Section~\ref{sec: non-aspherical}). Recall that $SP^n(M_g)$ is the orbit space of the permutation action of the symmetric group $\Sigma_n$ on the Cartesian product $(M_g)^n$.\label{eq: ex for SP}
\end{enumerate}
\vspace{3mm}
The smooth projective varieties $SP^n(M_g)$, and some of Gompf's branched covers of $\mathcal{KT}$, have infinite $\Pi$, and so they are ``non-trivial examples''. Several other such manifolds are known --- for different examples, see~\cite{Go2,IKRT}.
\end{ex}

\begin{remark}
    We note that symplectically aspherical K\"ahler manifolds are \emph{not} just symplectically aspherical manifolds that support a K\"ahler structure. It follows from~\cite{DCDJ} that the $n$-th symmetric product $SP^n(M_g)$ of a curve $M_g$, where $M_g$ has a \emph{hyperelliptic} complex structure, is \emph{not} symplectically aspherical K\"ahler for $n,g\ge 2$ since it contains a holomorphically embedded rational curve (on which any compatible K\"ahler form obviously evaluates non-trivially). However, $SP^n(M_g)$ is a smooth projective variety whose underlying manifold is symplectically aspherical for $g\ge 2n-1$ (\emph{cf}. (3) in Example~\ref{ex: sa and sak}).
\end{remark}

Let us look at basic properties of symplectically aspherical K\"ahler manifolds.

\begin{lemma}\label{lem: sak properties}
    Symplectically aspherical K\"ahler manifolds are stable under taking finite products, complex submanifolds, and finite covers.
\end{lemma}

\begin{proof}
    Let $(M,J_M,\omega_M)$ and $(N,J_N,\omega_N)$ be two symplectically aspherical K\"ahler manifolds, let 
    $\iota\colon (M',\iota^*J_M,\iota^*\omega_M)\hookrightarrow (M,J_M,\omega_M)$ be a complex submanifold, and let $p\colon N'\to N$ be a finite covering. First, $\iota^*\omega_M$ is an aspherical K\"ahler form on $M'$ by definition. Next, let $\pr_M\colon M\times N\to M$ and $\pr_N\colon M\times N\to N$ be the projections. Then it is easy to check that $\omega = \pr_M^*\omega_M+\pr_N^*\omega_N$ and $\omega'=p^*\omega_N$ are aspherical K\"ahler forms on $M\times N$ and $N'$, respectively, with respect to the naturally induced complex structures (see, for instance,~\cite[Proposition~2.1]{KRT2}).
\end{proof}

The existence of a symplectically aspherical form on a closed smooth manifold is rather constraining on its topology and geometry, as the following results suggest (see also Section~\ref{sec: fund grp} for additional constraints imposed, particularly on the fundamental group, in the presence of aspherical \emph{K\"ahler} forms).

\begin{lemma}\label{lem: essential}
If $(M,\omega)$ is a closed $c$-symplectically aspherical $2n$-manifold, then
\begin{enumerate}
    \item $[\omega]\in \Im(f^*\colon H^2(B\pi_1(M);\R) \to H^2(M;\R))$ for a map $f\colon M\to B\pi_1(M)$ that classifies the universal cover of $M$;

    \item $M$ is essential, or, equivalently, its Lusternik--Schnirelmann category is $2n$;

    \item $M$ satisfies the strong Arnol'd conjecture, provided $\langle c_1(M),\Pi(M)\rangle=0$ for the first Chern class of $M$, and $S^1$-actions on $M$ are never Hamiltonian.
\end{enumerate}
\end{lemma}

We refer the reader to~\cite{LO,RO} for topological proofs of Lemma~\ref{lem: essential}.

Actually, $(1)\implies (2)$ above since a closed orientable $k$-manifold $M$ is \emph{essential} (also called \emph{rationally essential}) if $f_*([M]_\R)\ne 0$ in $H_{k}(B\pi_1(M);\R)$ for a classifying map $f\colon M\to B\pi_1(M)$, see~\cite{Gr-ess} (and also~\cite{KR} for equivalent definitions). Also, a direct implication of $(1)$ for the cohomological dimension of $\pi_1(M)$ is that 
\begin{equation}\label{eq: cd}
\cd_\R(\pi_1(M))\ge 2n=\dim_\R(M), 
\end{equation}
since $[\omega]^n\ne 0$. Of course, the equality $\cd_\R(\pi_1(M))=2n$ holds when $M$ is a closed aspherical $2n$-manifold (i.e., has contractible universal cover), but the converse is far from being true --- see~\cite{KRT1} and Section~\ref{sec: fund grp}.  

Next, we mention that $\omega$ is an aspherical K\"ahler form on a closed manifold $M$ if and only if its pullback on the universal cover of $M$ is exact (note that $\omega$ cannot be exact on $M$ itself); this is easy to see from the basic properties of the second real Hurewicz map. In fact, the following is well-known 
(see, for example,~\cite{Bor}).

\begin{lemma}\label{lem: exact}
Let $M$ be a closed symplectic manifold, and let $\pi\colon\wt M\to M$ be the universal Riemannian covering. Then the following are equivalent:
\begin{enumerate}
    \item A symplectic form $\omega$ on $M$ is aspherical.
    \item The pullback symplectic form $\pi^*\omega$ on $\wt M$ is exact.
    \item If $f\colon (N,\omega')\to (M,\omega)$ is a symplectic map, where $(N,\omega')$ is a closed positive dimensional symplectic manifold, then $f_*(\pi_1(N))$ is infinite.
\end{enumerate}
\end{lemma}

The implications (1) $\iff$ (2) $\implies$ (3) use standard topological arguments; however, proving (3) $\implies$ (1) requires also the use of Gromov's $h$-principle for isosymplectic immersions (see~\cite{EM}). 

Lemma~\ref{lem: exact} implies the following (see also~\cite[Corollary~2.12]{TJLi1}). 

\begin{cor}\label{cor: no spheres}
If $(M,\omega)$ is a symplectically aspherical $2n$-manifold, then it does not contain any symplectically embedded $2$-spheres. In particular, any symplectically aspherical $4$-manifold is minimal. 
\end{cor}

Recall that a closed symplectic $4$-manifold $(M,\omega)$ is called \emph{minimal} if it does not contain any symplectically embedded $2$-spheres with self-intersection $-1$.

\section{Symplectically aspherical K\"ahler manifolds in complex geometry}\label{sec: in complex geom}

In this section, we look at some complex geometric aspects of symplectically aspherical K\"ahler manifolds and study how these manifolds interact with other classically studied closed K\"ahler manifolds, such as Kobayashi(--Brody) hyperbolic, Gromov's K\"ahler hyperbolic, and Koll\'ar large manifolds.

First, we describe a general method for constructing symplectically aspherical K\"ahler manifolds. Note that closed K\"ahler manifolds that are aspherical are symplectically aspherical K\"ahler for obvious reasons (\emph{cf}. Section~\ref{sec: prelim}), and their examples are easy to come by (see, for instance,~\cite{ABCKT,IKRT}). Also, several methods for constructing \emph{non-aspherical} symplectically aspherical manifolds are known in the literature (see, for example,~\cite{Go2,IKRT,FMS,KRT1} and the references therein). In that spirit, our method also yields non-aspherical 
examples of symplectically aspherical K\"ahler manifolds in each complex dimension $\ge 2$. 

For this purpose, below we use the standard cyclic covering construction outlined in~\cite[Section~I.17]{bpv}.

\begin{thm}\label{th: new construction}
Let $X^n$ be the $k$-cyclic branched cover of an aspherical smooth projective $n$-variety $Y^n$, branched along a smooth ample divisor $B^{n-1}$ determined by a line bundle $L$ such that $L^{\otimes k}=\OO_{Y^n}(B^{n-1})$. Then $X^{n}$ is a symplectically aspherical smooth projective $n$-variety for $n\ge 1$,
but it is not aspherical for $n\ge 2$.
\end{thm}

\begin{proof}
The varieties $X^n$ are constructed as follows. Start with an ample line bundle, say $L$, on $Y$ and a smooth divisor $B$ in the linear system $|kL|$. It is standard to construct a degree $k$ finite morphism of projective varieties $f\colon X\to Y$, which we call the {\em $k$-cyclic cover given by the relation $kL\sim B$} (here, $\sim$ denotes linear equivalence). The variety $X$ is smooth, since  $B$ is, and standard formulae for cyclic branched covers (see, for instance,~\cite[Lemma~17.1]{bpv}) give
\[
K_X=f^*(K_Y\otimes L^{\otimes (k-1)})=f^*(K_{Y})\otimes \OO_{X}(R),
\]
where $K_X$ denotes, as usual, the canonical bundle, and $R$ the ramification divisor. Recall that the canonical line bundle $K_{Y}$ is \emph{nef} for an aspherical smooth projective variety $Y$ (see, for example, the proof of Theorem~8.2 in~\cite{DCDJ}), and that the sum of a nef divisor with an ample one is ample (see, for instance,~\cite[Chapter~I]{Laz}). This implies that $K_{X}$ is ample, since $K_Y\otimes L^{\otimes (k-1)}$ is ample and the map is finite (see again~\cite[Chapter I]{Laz}). Since $K_X$ and $K_Y\otimes L^{\otimes (k-1)}$ are ample, we can find K\"ahler classes $\omega_X\in c_1(K_X)$ and $\omega_Y\in c_1(K_Y\otimes L^{\otimes (k-1)})$ such that $\omega_X=f^*\omega_Y$. Now, the fact that $\omega_X$ is an aspherical form follows from standard cohomological computations since $\omega_Y$ is an aspherical form (as $Y$ is an aspherical variety), see~\cite[Proposition 2.1]{KRT2}. Finally, to see that $X^n$ is not aspherical for any $n\ge 2$, note from~\cite[Proposition~3.3]{PT} that there is a central extension
\begin{equation}\label{eq: pardini}
1\to N\to \pi_1(X)\to \pi_1(Y)\to 1 
\end{equation}
for some \emph{finite} group $N$. The non-asphericity of $X$ is straightforward if $N$ is non-trivial, and it follows from~\cite[Theorem~2.2~and~Remark~2.3]{DCP} otherwise.
\end{proof}

Next, recall from Corollary~\ref{cor: no spheres} that a symplectically aspherical K\"ahler surface is minimal. Actually, we can say more for \emph{non-aspherical} such manifolds!

\begin{prop}\label{general}
Let $(M, J, \omega)$ be a symplectically aspherical K\"ahler surface that is not aspherical. Then $M$ is of general type (i.e., has maximal Kodaira dimension).
\end{prop}
\begin{proof}
First, note that $M$ cannot have negative Kodaira dimension, as this would imply that it is either rational or ruled. These manifolds contain rational curves, and as such cannot be symplectically aspherical K\"ahler (\emph{cf}. Corollary~\ref{cor: no spheres}). By the Kodaira--Enriques classification, surfaces with zero Kodaira dimension are covered either finitely by a $K3$ surface or infinitely by $\C^2$. Thus, if $M$ is non-aspherical with $\text{Kod}(M)=0$, then $\pi_1(M)$ is finite and so $M$ cannot be symplectically aspherical K\"ahler by the Hurewicz theorem. If $M$ is non-aspherical with $\text{Kod}(M)=1$, then by~\cite[Proposition~2.1]{ADL}, it is a properly elliptic surface with exceptional fibers containing rational curves. But then again $M$ cannot be symplectically aspherical K\"ahler. Thus, we must have $\text{Kod}(M)=2$.
\end{proof}

\begin{remark}
Proposition~\ref{general} does not generalize to higher dimensions starting from threefolds. Indeed, the product of the symmetric square $SP^2(M_g)$ with the complex torus $\C/\Lambda$ has Kodaira dimension $2$ (so is \emph{not} of general type), even though it is non-aspherical and symplectically aspherical K\"ahler for $g\geq 3$ when the Riemann surface $M_g$ has a generic complex structure (\emph{cf}. Example~\ref{ex: sa and sak}).
\end{remark}

Finally, we proceed to compare the class of symplectically aspherical K\"ahler manifolds with other well-studied classes of closed K\"ahler manifolds. 

Let us recall that a closed K\"ahler manifold $(M,J,\omega)$ 
\begin{enumerate}
    \item is \emph{real hyperbolic} if and only if $\pi_2(M)=0$ and $\pi_1(M)$ is a hyperbolic group, in the sense of~\cite{Gr-hyperbolic grp};
    \item has \emph{Stein universal cover} $\wt M$ if and only if $\wt M$ can be biholomorphically embedded as a closed complex subspace of $\C ^N$ for some positive integer $N$;
    \item is \emph{Kobayashi hyperbolic} if and only if every $J$-holomorphic map $f\colon\C \to M$ is constant, see~\cite{Brody};
    \item is \emph{K\"ahler hyperbolic} if and only if the pullback K\"ahler form $\pi^*\omega$ on $\wt M$ along the universal covering $\pi\colon(\wt M,\wt g)\to (M,g)$ is $d$\emph{-bounded} (i.e., $\pi^*\omega=d\omega'$ for a closed $1$-form $\omega'$ that is bounded with respect to $\wt g$), see~\cite{Gr-kahler hyperbolic};
    \item has \emph{large fundamental group} if and only if the universal cover $\wt M$ does not contain positive dimensional compact complex subspaces~\cite{Kollar1,Kollar2} (for brevity, we call such manifolds \emph{Koll\'ar large}).
\end{enumerate}

It is natural to study relationships between the above K\"ahler manifolds and symplectically aspherical K\"ahler manifolds, especially those of general type (which include, in particular, \emph{all} non-aspherical symplectically aspherical K\"ahler surfaces, see Proposition~\ref{general}). 
The latter have an ample canonical line bundle and admit a K\"ahler--Einstein metric of negative Ricci curvature, see~\cite[Theorem~2.8]{PingLi} and also~\cite[Section~5]{DCDJ} for illustrations of these results using the symmetric products $SP^n(M_g)$ for $g\ge 2n-1\ge 3$ in the case when $M_g$ is a complex curve generic in its moduli space. 
It is also worth noting from~\cite[Corollary~1.7]{CY2} that a symplectically aspherical K\"ahler manifold having a hyperbolic fundamental group is K\"ahler hyperbolic, so due to Gromov~\cite{Gr-kahler hyperbolic}, it is of general type and satisfies the Hopf--Singer conjecture on the $L^2$-Betti numbers of the universal cover.


\begin{thm}\label{prop: other kahler}
    We have the following relationships between the above classes of closed K\"ahler manifolds (here, ``SAK'' stands for symplectically aspherical K\"ahler):
    \[
    \begin{tikzcd}[row sep=2.6em, column sep=3em, arrows=Rightarrow, every arrow/.append style={shift left=0.8ex}]
    \fbox{\Centerstack[c]{Kobayashi\\ hyperbolic}} \arrow[swap,negated]{rrr}{(d)} 
    & & & 
    \fbox{\Centerstack[c]{Koll\'ar\\ large}}  \arrow[dashed]{dd}{??}
    \\
    \fbox{\Centerstack[c]{K\"ahler \\ hyperbolic}} \arrow{u}{(2)} \arrow[negated]{d}{{(a) \vphantom{1}}} \arrow{r}{(3)}
    &
    \fbox{\Centerstack[c]{SAK of \\ general type}} \arrow[negated]{l}{{(b) \vphantom{1}}} \arrow{r}{(4)} \arrow[swap,negated]{ul}{(c)}
    &
    \fbox{\Centerstack[c]{SAK}} \arrow[dashed]{dr}{?}  \arrow[negated]{l}{{(e) \vphantom{1}}} \arrow{ur}{(5)} 
    &
    \\
    \fbox{\Centerstack[c]{Real \\hyperbolic}}  \arrow{u}{(1)}
    & & &
    \fbox{\Centerstack[c]{Stein \\ universal \\ cover}} \arrow{uu}{(6)} 
    \arrow[negated]{ull}{{(f) \vphantom{1}}}
    \end{tikzcd}  
    \]
\end{thm}

\begin{proof} (1) and (2) are well-known, see~\cite{Gr-kahler hyperbolic} (and also~\cite[Corollary~4.2]{CY}), and (4) is obvious. 

\begin{enumerate} [(1)]
    \item [(3)] Note that a K\"ahler hyperbolic manifold is of general type~\cite{Gr-kahler hyperbolic}, and a $d$-bounded form on its universal Riemannian cover is exact by definition. 

    \item [(5)] Let $(M,J,\omega)$ be SAK. The proof that it is Koll\'ar large is a standard volume computation using Stokes's theorem and is well-known to specialists. Let $\pi\colon \wt M\to M$ be the universal covering, and assume that there exists a compact complex subspace $N\subset \wt M$ with $\dim_\R(N)=2n\ge 2$. Then the volume of $N$ with respect to the K\"ahler metric $\pi^*\omega$ is positive. However, the exactness of $\pi^*\omega$ (say $\pi^*\omega=d\omega'$ for a suitable $1$-form $\omega'$) implies that this volume is equal to
\[
\int_N\ \frac{1}{n!}\ (\pi^*\omega)^n=\int_N\ \frac{1}{n!}\ d(\omega'\land (\pi^*\omega)^{n-1}),
\]
and the latter is $0$ since $N$ is compact. This is a contradiction. 

    \item [(6)] If the universal cover is Stein, then it does not contain positive dimensional compact complex subspaces, so the manifold is Koll\'ar large by definition.
\end{enumerate}

For (e) and (f), note that the complex tori $\C^n/\Lambda$ with their standard K\"ahler structure are SAK (because they are aspherical) and have Stein universal cover $\C^n$, but they are not of general type. 

\begin{enumerate}[(a)]
    \item A projective curve $M_g$ of genus $g\ge 2$ is real hyperbolic and hence K\"ahler hyperbolic by (1). The product $M_g\times M_g$ is obviously K\"ahler hyperbolic, but $\pi_1(M_g\times M_g)$ contains a copy of $\Z^2$, and therefore, $\pi_1(M_g\times M_g)$ is not a hyperbolic group, see~\cite{Gr-hyperbolic grp}. Hence, $M_g\times M_g$ is not real hyperbolic.

    \item The $n$-th symmetric product $SP^n(M_g)$ of a generic complex curve $M_g$ of genus $g\ge 2n-1$ is SAK (\emph{cf}.~\eqref{eq: ex for SP} in Example~\ref{ex: sa and sak}), of general type~\cite[Page~39]{Abr}, and has amenable fundamental group $\pi_1(SP^n(M_g))\cong H_1(M_g)\cong\Z^{2g}$. But the fundamental group of a K\"ahler hyperbolic manifold cannot be amenable (for a proof, see~\cite[Theorem~6.7]{Ked}).

    \item A smooth toroidal compactification $\ov{\C \mathcal{H}^2/\Gamma}$ of finite volume is a closed K\"ahler surface containing a holomorphic elliptic curve, and hence an entire curve, and so it can never be Kobayashi hyperbolic. On the other hand,~\cite[Theorem~A]{DC} produces smooth toroidal compactifications of general type that admit Riemannian metrics of non-positive sectional curvature, and hence are SAK of general type.

    \item By a result of Siu, a very generic surface $X$ in $\C P^3$ having sufficiently large degree is Kobayashi hyperbolic, see~\cite{Siu} (and also~\cite[Corollary~1]{DEG}). Of course, $\pi_1(X)\cong\pi_1(\C P^3)\cong 0$ due to the Lefschetz hyperplane theorem. But the fundamental group of a Koll\'ar large manifold must be infinite.
\end{enumerate}

The implications (?) and (??) are \emph{predicted} by the Shafarevich conjecture.
\end{proof}

\begin{remark}
It is currently unknown whether any Koll\'ar large variety must have a Stein universal cover. Since symplectically aspherical K\"ahler manifolds constitute an important (and topologically tame) subclass of Koll\'ar large manifolds (\emph{cf}. implication (5) in Proposition~\ref{prop: other kahler}), a step towards settling the Shafarevich conjecture could be to first check it on symplectically aspherical K\"ahler manifolds (i.e., to check whether the implication (?) holds)! We note that several of these manifolds are known to have Stein universal covers. Examples include
\begin{enumerate}
    \item aspherical ones, such as complex tori, Hermitian locally symmetric spaces, Kodaira surfaces, etc., and
    \item non-aspherical ones, such as the symmetric products of curves $SP^n(M_g)$ for $g\ge 2n-1$ for $M_g$ generic (see~\cite[Proposition~5.5]{DCDJ}), and the smooth projective varieties $X^n$ constructed in Theorem~\ref{th: new construction}, \emph{provided} the corresponding aspherical varieties $Y^n$ have Stein universal cover (this is because an unramified covering of a Stein manifold is Stein).
\end{enumerate}
\end{remark}

\begin{remark}
    In~\cite{JZ,CX}, K\"ahler \emph{parabolic} or \emph{non-elliptic} manifolds were introduced as generalizations of K\"ahler hyperbolic manifolds. The former have a K\"ahler form whose pullback to the universal cover has $d$\emph{-sublinear growth}. Certainly, the class of SAK manifolds contains the class of these closed manifolds as well.
\end{remark}


We (the authors) do not know if the reverse implication of (5) in Proposition~\ref{prop: other kahler} holds in non-trivial cases, so we conclude this section with the following question.
\begin{question}\label{ques: kollar large}
 If $(M,J,\omega)$ is a Koll\'ar large manifold with $\Pi(M)\ne 0$, then is it true that $M$ admits a $J$-compatible K\"ahler form that is aspherical?
\end{question}

\section{The fundamental group and the realizability problem}\label{sec: fund grp}

The existence of an aspherical K\"ahler form on a closed smooth manifold severely constrains its fundamental group (see below). In this section, we show that each free abelian group that \emph{can be} realized as the fundamental group of a symplectically aspherical (smooth) projective variety with non-trivial second homotopy group \emph{is} indeed realized as such in complex dimension $2$.

We call a finitely presented discrete group $\Gamma$  an \emph{SA group} (resp. \emph{SAK group}) if $\Gamma$ is isomorphic to the fundamental group of a closed symplectically aspherical manifold (resp. symplectically aspherical projective variety). The class of SAK groups is closed under direct products and passage to finite-index subgroups (\emph{cf}.  Lemma~\ref{lem: sak properties}). The fundamental groups of closed K\"ahler manifolds (namely the \emph{K\"ahler groups}) and SA groups have been studied in detail in, for instance,~\cite{ABCKT} and~\cite{IKRT,KRT1}, respectively --- see also the references therein. Here, we study SAK groups.

If $\Gamma$ is an SAK group, then
\begin{enumerate}
\item $\Gamma$ must be infinite with exactly one end (so $\Gamma$ cannot be a free product);
\item the first Betti number $b_1(\Gamma')$ must be even for any finite index subgroup $\Gamma'$ of $\Gamma$ (so $b_1(\Gamma)$ cannot be odd);
\item if $b_1(\Gamma) \ne 0$, then the second Betti number $b_2(\Gamma)$ must be non-zero as well;
\item if $\Gamma\not\cong\pi_1(M_g)$ for a closed orientable surface $M_g$ of genus $g > 0$, then there must exist $x\in H^2(\Gamma;\R)$ such that $x^2\ne 0$ (so $\cd_{\R}(\Gamma)\ne 3$); 
\item if $\Gamma\not\cong\Z^2$ is a finitely generated abelian group, then $\r(\Gamma)$ must be an even integer $\ge 4$.
\end{enumerate}

\begin{remark}\label{rem: proper inclusion}
The class of SAK groups is strictly smaller than that of both SA groups and K\"ahler groups. Indeed, $\Z^{2m-1}\oplus G$ is an SA group for any finite abelian group $G$ due to~\cite{KRT1}, but not SAK, and all finite groups are K\"ahler due to the work of Serre (see~\cite[Example 1.11]{ABCKT}), but they are also not SAK.     
\end{remark}

The free abelian groups $\Z^{2n}$ for $n\ge 1$ are SAK groups: their corresponding smooth projective varieties are the complex tori $\C^n/\Lambda$. The ``realizability problem'' of the fundamental group becomes interesting when one asks whether an SAK group is the fundamental group of a symplectically aspherical projective variety with \emph{non-trivial} second homotopy group! So, it makes sense to study such SAK groups separately. For concreteness, let us say that
\begin{enumerate}
    \item $\Gamma\in \A'_n$ if there exists a symplectically aspherical projective $n$-variety $M$ with $\pi_1(M)=\Gamma$ and $\pi_2(M)= 0$;
    \item  $\Gamma\in \B'_n$ if there exists a symplectically aspherical projective $n$-variety $M$ with $\pi_1(M)=\Gamma$ and $\pi_2(M)\ne 0$.
\end{enumerate}
Without the prime in the notation, the corresponding subclasses $\A_n$ and $\B_n$ of SA groups were studied in~\cite{IKRT}. Finally, let $\A'=\bigcup_{n\ge 1}\A'_n$ and $\B'=\bigcup_{n\ge 1}\B'_n$.

We first note that the realizability problems of the fundamental group can be reduced to complex dimension $2$.

\begin{prop}\label{prop: dimension shifting}
If $\Gamma$ is the fundamental group of a symplectically aspherical projective $n$-variety $(M,J,\omega)$ for some $n \ge 3$, then $\Gamma$ is the fundamental group of a symplectically aspherical projective $(n-i)$-variety for each $1 \le i \le n-2$.
\end{prop}

\begin{proof}
    In view of induction, it suffices to prove that $\Gamma$ can be realized as the fundamental group of a closed symplectically aspherical projective $(n-1)$-variety $N$. To that end, take a positive line bundle $L$ over $M$, which has a connection with curvature $-i\omega$, such that its first Chern class $c_1(L)$ lifts $[\omega/2\pi]$ to an integral class. Indeed, the K\"ahler form represents an integral cohomology class because $M$ is a projective variety. It is a standard and classical result in complex geometry that for a sufficiently large positive integer $k$, the tensor product $L^{\otimes k}$ has holomorphic sections, and the zero set of a generic section of $L^{\otimes k}$ is a hyperplane section $N$ of $M$ that realizes the Poincar\'e dual of $k\cdot c_1(L)$. If $\phi\colon N\hookrightarrow M$ denotes the inclusion, then $\phi^*\omega$ is an aspherical K\"ahler form on $N$ by definition. By the Lefschetz hyperplane theorem, $\phi$ is a $(n-2)$-equivalence and so, $\pi_1(N)\cong\pi_1(M)=\Gamma$.
\end{proof}

\begin{cor}\label{cor: inclusion of SAK groups}
    We have the inclusions $\A'_{i+1}\subset \A'_{i}$ for each $i\ge 3$, and $\B'_{i+1}\subset \B'_{i}$ for each $i\ge 2$. Also, $\A'_3\subset \A'_2\cup \B'_2$.
\end{cor}

\begin{proof}
Let $M$ be a symplectically aspherical projective $(i+1)$-variety such that $\pi_1(M)=\Gamma$. Taking $n=i+1$ in the notations of the proof of Proposition~\ref{prop: dimension shifting}, we get an $(i-1)$-equivalence $\phi:N^{i}\to M^{i+1}$. If $i \ge 3$, then $\pi_j(N) \cong \pi_j(M)$ for $j\in\{1,2\}$. Thus, $\A'_{i+1}\subset \A'_{i}$ and $\B'_{i+1}\subset \B'_{i}$ for $i\ge 3$. The other two inclusions follow upon observing that $\pi_1(N^2)\cong \pi_1(M^3)$, and that $\pi_2(M^3)$ is isomorphic to a subgroup of $\pi_2(N^2)$. 
\end{proof}

Next, we check which \emph{free abelian} groups belong to class $\B'$. Note that $\Z^2\not\in \B'$; this can be seen using~\eqref{eq: cd}.  
Thus, $\A'\not\subset\B'$, but it is unclear whether $\B'\subset \A'$.

\begin{theorem}\label{th: main1}
Each free abelian group of even rank $\ge 4$ is realized as the fundamental group of a closed symplectically aspherical smooth projective surface of general type that has non-trivial real spherical homology classes. 
\end{theorem}

\begin{proof}
    For a free abelian group $\Z^{2g}$ of rank $2g\ge 6$, the symmetric square $SP^2(M_g)$ of a generic curve $M_g$ of genus $g\ge 3$ is the desired surface. Indeed, we have that $\pi_1(SP^2(M_g))\cong H_1(M_g)\cong\Z^{2g}$, and it is explained in~\cite{DCDJ} that $\Pi(SP^2(M_g))$ is infinite. Moreover, $SP^2(M_g)$ holomorphically embeds inside the Jacobian $\C^g/\Lambda$ via the Abel--Jacobi map $\mu_2\colon SP^2(M_g)\to \C^g/\Lambda$ for $g\ge 3$, see~\cite[Theorem~4.1]{DCDJ}. Thus, the pullback $\mu_2^*\omega_0$ of the standard  aspherical K\"ahler form $\omega_0$ on $\C^g/\Lambda$ is an aspherical K\"ahler form on $SP^2(M_g)$. That $SP^2(M_g)$ is of general type for $g\ge 3$ is well-known (see, for example,~\cite[Page~39]{Abr}).

    The above approach does not work for $g=2$: indeed, $SP^2(M_2)$, being diffeomorphic to the symplectic blow-up of $T^4$ at a point, is not minimal and hence not symplectically aspherical due to Corollary~\ref{cor: no spheres}. So for $\Z^4$, we proceed as follows. Consider a minimal surface of general type on the Severi line --- remarkably, such surfaces are completely classified in~\cite{Barja}. Each of these surfaces is the minimal resolution of a double cover of an abelian surface $A$ branched along an ample divisor $B$. When this divisor is \emph{smooth}, the double branch cover, say $\pi:X\to A$, is smooth and has an ample canonical divisor $K_{X}=\pi^* B$. For more details on this construction, see~\cite{Wang} (and also~\cite[Page~3]{DCP}). Therefore, proceeding as in the proof of Theorem~\ref{th: new construction}, we can find an aspherical K\"ahler class $\omega\in c_{1}(K_{X})$. Also, we deduce from~\cite[Corollary~3.4]{PT} that $\pi_1(X)\cong\pi_1(A)\cong\Z^4$ (indeed,  the finite group $N$ from the central extension~\eqref{eq: pardini} turns out to be trivial in this special case). Finally, as explained in~\cite[Corollary~3.2]{DCP}, surfaces constructed in this way have infinite $\Pi$, for there is a singular surface in the connected component of $X$ in the moduli space of the canonical models of surfaces.
\end{proof}

This implies that $\Z^{2n}\in \A'_{n}\cap \B'_{n}$ for all $n\ge 2$. Indeed, if $\Gamma \in \A'_\ell$ and $\Gamma'\in \B'_m$ for some $\ell,m\ge 1$, then $\Gamma\oplus\Gamma'\in \B'_{\ell+m}$, so we get $\Z^{2n}\in \B'_{n}$ for each $n \ge 3$ as well since $\Z^4\in \B'_2$ (\emph{cf}. Theorem~\ref{th: main1}) and $\Z^{2n-4}\in \A'_{n-2}$.


\begin{remark}
The projective varieties described in the proof of Theorem~\ref{th: main1} have Stein universal cover. For the surface $X$, see~\cite{Wang} for this fact, and for the symmetric squares $SP^2(M_g)$ for $g\ge 3$, see~\cite[Proposition~5.5]{DCDJ}.
\end{remark}

The following is then a direct consequence of Theorem~\ref{th: main1}.

\begin{cor}
     Each even integer $\ge 4$ is realized as the first Betti number of a closed symplectically aspherical smooth projective surface of general type that has non-trivial real spherical homology classes and a Stein universal cover. 
\end{cor}

\begin{remark}
The question of whether $\Z^4\in\B_2$ was asked in~\cite{IKRT} and answered in the affirmative in~\cite{KRT1}. Unlike the examples of~\cite{KRT1} (for which it is unclear whether $\Pi$ is a non-trivial subspace), Theorem~\ref{th: main1} produces explicit examples with $\Pi$ infinite to show that $\Z^4\in\B_2'\subsetneq \B_2$.
\end{remark}

Inspired by the results of~\cite{KRT1}, that $\Z^k\oplus G$ are SA groups (in fact, in $\B$) for any finite abelian group $G$, we ask the question of whether $\Z^{2n}\oplus G$ is an SAK group (preferably in $\B')$  for each finite abelian group $G$. 

\begin{question}\label{ques: grp}
For each even integer $k\ge 4$ and finite abelian group $G$, does there exist a closed symplectically aspherical K\"ahler manifold $M$ with $\pi_2(M)\ne 0$ such that $\pi_1(M)\cong\Z^k\oplus G$?
\end{question}

In particular, it is not clear to us whether the fundamental group of a symplectically aspherical K\"ahler manifold in real dimensions $\ge 4$ can have non-trivial torsion. Note that Theorem~\ref{th: main1} answers Question~\ref{ques: grp} for trivial $G$. 

\section{Symplectically apherical manifolds and scalar curvature}\label{sec: psc}

In this section, we propose (and provide support towards) an extension of the following long-standing conjecture.

\begin{conjec}[Gromov--Lawson]\label{GLC}
A closed aspherical manifold cannot support a Riemannian metric of positive scalar curvature. 
\end{conjec}

We argue that it is very interesting to consider the following strengthening of Conjecture~\ref{GLC} in the realm of symplectic manifolds.

\begin{conjec}\label{DDJ}
A closed symplectically aspherical manifold cannot support a Riemannian metric of positive scalar curvature. 
\end{conjec}

%
%

The main idea behind such conjectures is that extra enriched structures on smooth manifolds usually prevent the positivity of their curvature. 

The first piece of evidence towards this conjecture comes from the theory of the Novikov conjecture. Indeed, the following may be known to experts.

\begin{thm}\label{th: c-symplectic psc}
A closed $c$-symplectically aspherical $2n$-manifold that has a spin finite cover cannot support a Riemannian metric of positive scalar curvature. In particular, Conjecture~\ref{DDJ} is true in the spin case.
\end{thm}

\begin{proof}
The proof is based on the fact that $2$-dimensional cohomology classes satisfy the Novikov conjecture (see~\cite{CGM,M,HS}), and on Rosenberg's index theorem for complex $K$-theory, which states that the index of a closed spin manifold $M$, with fundamental group $\pi$, that admits a Riemannian metric of positive scalar curvature is zero in $K_*(C^*_{\max}\pi)$. Here, $C^*_{\max}\pi$ denotes the maximal $C^*$-algebra of $\mathbb C\pi$.
Nowadays, Rosenberg's classical index theorem is stated in the literature for the $KO$-theory~\cite{RS}, which is a more delicate statement than for complex $K$-theory.
So for the complex version of it, we refer to Rosenberg's first publication on this subject~\cite{R}, where the index theorem was proven for complex $K$-theory, but was not stated in terms of an index map. The complex version can be easily derived from the real one by means of the complexification functor $KO\to KU$.

Let $M$ be a spin finite cover of our manifold. It suffices to show that $M$ does not admit a Riemannian metric of positive scalar curvature. Write $\pi:=\pi_1(M)$. Note that $M$ has an aspherical $c$-symplectic form $\omega$ (\emph{cf}. Lemma~\ref{lem: sak properties}). Since $\omega$ is aspherical, there exists $\gamma\in H^2(B\pi;\R)$ such that $f^*(\gamma)=[\omega]$ for a classifying map $f\colon M\to B\pi$. Since $\omega$ is also $c$-symplectic, we have $f^*(\gamma)^n=[\omega]^n\ne 0$ and so $f_*([M]_\R)\ne 0$ in $H_{2n}(B\pi;\R)$ by Poincar\'e duality. For these facts, we refer the reader to the review in Section~\ref{sec: prelim}. It then follows from the Chern character isomorphism 
\[
\ch\colon K_*(B\pi)\otimes\mathbb \R\to H_*(B\pi;\mathbb R)
\] 
that $\ch(f_*([M]_K))=f_*([M]_\R)\ne 0$. Let $\ch'$ be the homomorphism in the following commutative diagram:
\[
\begin{tikzcd}[row sep=2.5em]
    K_*(B\pi)\otimes\R \arrow[r,"\ch","\cong"']
    &
    H_*(B\pi;\R)
    \\
    K_*(M)\otimes\R \arrow{r}{\ch'} \arrow{u}{f_*\otimes\id_\R}
    &
    H_*(M;\R). \arrow[swap]{u}{f_*}
\end{tikzcd}
\]
In particular, $\ch'([M]_K)\ne 0$. Now, note that
\[
    \langle \gamma^n, \ch(f_*([M]_K))\rangle  =  \langle \gamma^n, f_*(\ch'([M]_K))\rangle =  \langle f^*(\gamma)^n, \ch'([M]_K)\rangle \ne 0
\]
 in view of the Poincar\'e duality on $M$.
But then we can use~\cite[Theorem~1.2]{HS} to deduce that $A(f_*([M]_K))\ne 0$, where
\[
A\colon K_*(B\pi)\to K_*(C^*_{\max}\pi)
\]
is the Baum--Connes assembly map.  Finally, we consider the index map 
\[
\mathcal{I}\colon K_*(M)\to K_*(C^*_{\max}\pi).
\]
Since $\mathcal{I}=A\circ f_*$, we get $\mathcal{I}_*([M]_K)\ne 0$. This gives an obstruction to metrics of positive scalar curvature on $M$ in view of Rosenberg's index theorem~\cite{R}.
\end{proof}

We believe that the hypothesis in Theorem~\ref{th: c-symplectic psc}, that the $c$-symplectically aspherical manifold admits a finite spin cover, can be weakened to the hypothesis that the universal cover of the manifold admits a spin structure (note that this will settle the $c$-symplectic case of the Gromov--Lawson Conjecture~\ref{GLC}). 
It was proven in~\cite[Proposition~4.2]{Dr} that in the case of \emph{2-tame fundamental groups}, the latter implies the former. Here, a group $\pi$ is called 2-tame if there exists a finite index subgroup $\Gamma$ of $\pi$ such that the homomorphism $H^2(\pi;\mathbb Z_2)\to H^2(\Gamma;\mathbb Z_2)$ induced by the inclusion $\Gamma\hookrightarrow \pi$ is trivial, see~\cite{Dr}. Since finitely generated abelian groups are 2-tame, we obtain the following.

\begin{cor}\label{cor: abelian fund grp}
A closed $c$-symplectically aspherical manifold having an abelian fundamental group and a spin universal cover cannot support a Riemannian metric of positive scalar curvature. 
\end{cor}

The next piece of evidence towards Conjecture~\ref{DDJ} comes from low-dimensional topology. Indeed, the conjecture is always true in dimensions $2$ and $4$.

\begin{thm}\label{th: 2 and 4}
Closed symplectically aspherical $2$- and $4$-manifolds cannot support Riemannian metrics of positive scalar curvature. 
\end{thm}

\begin{proof}
In dimension $2$, a closed symplectically aspherical manifold must be aspherical. Thus, the conclusion follows from the classical Gauss--Bonnet formula. 

In dimension $4$, we employ low-dimensional topology techniques, in particular, Seiberg--Witten theory. First, recall that a celebrated result of Taubes~\cite[Main~Theorem]{Taubes} says that a symplectic $4$-manifold $M$ with $b^+_{2}(M)\geq 2$ has non-vanishing Seiberg--Witten invariants. Under these assumptions, a Lichnerowicz-type argument implies that $M$ cannot support a Riemannian metric with positive scalar curvature. So, it only remains to study the case $b^+_{2}(M)=1$. Observe that, as noted in Corollary~\ref{cor: no spheres}, a symplectically aspherical $4$-manifold $M$ cannot have a symplectically embedded sphere, which implies that $M$ is minimal. Due to the results of Liu and Ohta--Ono (see, for example,~\cite[Corollary 1.4]{McDuff}), a minimal symplectic $4$-manifold $M$ admits a metric of positive scalar curvature if and only if it is either rational or ruled. Now, if $M$ is rational, then it is simply connected, and if $M$ is ruled, then its fundamental group is a surface group. In any case, we get $\cd_\R(\pi_1(M))<4$, which contradicts~\eqref{eq: cd} since $M$ is symplectically aspherical.
\end{proof}

Finally, we prove a strengthening of Conjecture~\ref{DDJ} for the scalar curvature associated with a K\"ahler metric in each complex dimension $\ge 2$.

\begin{thm}\label{th: asphericalK}
A smooth projective $n$-variety $M$ with a birational morphism onto a symplectically aspherical smooth projective $n$-variety $N$ cannot support a K\"ahler metric of positive scalar curvature. 
\end{thm}

\begin{proof}
The proof of this statement follows closely the proof of Theorem~8.2 from \cite{DCDJ}. First, since $N$ is symplectically aspherical K\"ahler, its universal cover has no positive dimensional compact complex subvarieties (\emph{cf.} (5) in Proposition~\ref{prop: other kahler}). Thus, $N$ cannot support rational curves, and so by Mori's cone theorem~\cite[Theorem~1.24]{Mori}, its canonical line bundle $K_N$ is nef. From here onwards, the proof is the same as in our previous theorem,~\cite[Theorem~8.2]{DCDJ}. In detail, it follows that $K_N$ (and hence $K_M$, because $K_M=\pi^*K_N+R$ for a birational morphism $\pi\colon M\to N$ and an effective divisor $R$ on $M$) is pseudo-effective and is therefore equipped with
a singular Hermitian metric whose curvature is a closed positive current. It can then be shown using standard curvature computations that this obstructs the existence of a compatible K\"ahler metric of positive scalar curvature on $N$ (and similarly on $M$). 
\end{proof}

\begin{ex}\label{ex: no psc}
    Consider the symplectically aspherical K\"ahler symmetric products $SP^n(M_g)$ for $g\ge 2n-1\ge 3$. Theorem~\ref{th: asphericalK} implies that no such $SP^n(M_g)$ admits a K\"ahler metric of positive scalar curvature. For $n-g$ odd, since $SP^n(M_g)$ has a spin finite cover (see~\cite[Theorem~9.12]{DCDJ}), there are no such Riemannian metrics as well by Theorem~\ref{th: c-symplectic psc}.
\end{ex}
Since the universal cover of a Koll\'ar large smooth projective variety does not contain positive dimensional compact complex subspaces~\cite[Proposition~2.12]{Kollar1}, the above proof of Theorem~\ref{th: asphericalK} recovers the following fact.

\begin{cor}
If $M$ is a Koll\'ar large smooth projective variety (in particular, if $M$ is a smooth projective variety having a Stein universal cover), then $M$ cannot support a K\"ahler metric of positive scalar curvature.
\end{cor}

%
%
%
%
%
%

\begin{remark}\label{ex: not max mac dim}
In~\cite[Theorem~8.2]{DCDJ}, we proved Theorem~\ref{th: asphericalK} only in the special case when $N$ is an aspherical smooth projective variety (so a strengthening of Conjecture~\ref{GLC} in the K\"ahler setting). There, we also concluded for Gromov's \emph{macroscopic dimension} of (the universal Riemannian cover of) $M$ that $\dim_{mc}\wt M=2n$, using the fact that $\dim_{mc}\wt N=2n$ due to the asphericity of $N$. Here, $\dim_{mc}\wt M$ is the least non-negative integer $k$ for which there exists a $k$-dimensional simplicial complex $K$ and a continuous proper map $f\colon \wt M\to K$ such that $\text{diam}(f^{-1}(y))<C$ for some fixed $C>0$ and all $y\in K$, see~\cite{Gr-mac dim,Dr13}. However, as explained in~\cite[Section~10.1]{DCDJ}, the symplectically aspherical K\"ahler manifolds $SP^n(M_g)$, for $g\ge 2n-1$, need not have even one less than the maximum macroscopic dimension. Hence, it is \emph{not} true that a variety $M$ as in Theorem~\ref{th: asphericalK} has (one less than the) maximum macroscopic dimension $\dim_{mc}$. Therefore, even after \emph{assuming} Gromov's conjecture on the falling of the macroscopic dimension of manifolds with positive scalar curvature~\cite{Gr-mac dim}, 
it is unclear if there are any obvious obstructions to the existence of \emph{Riemannian} metrics of positive scalar curvature on $M$ in general. 
\end{remark}

On the other hand, we do have the following conclusion for the \emph{modified macroscopic dimension} $\dim_{MC}$, which, for the universal cover $\wt M$ of a closed Riemannian manifold $M$, is defined as the least non-negative integer $k$ for which there exists a $k$-dimensional simplicial complex $K$ and a map $f\colon \wt M\to K$ as in the definition of $\dim_{mc}$ above that is also Lipschitz, see~\cite{Dr11a}.

\begin{prop}
    If $M$ and $N$ are as in Theorem~\ref{th: asphericalK}, and if $\pi_1(N)$ is amenable, then $\dim_{MC}\wt M=\dim_{MC}\wt N=2n$.
\end{prop}

\begin{proof}
Since $N$ is symplectically aspherical, it is essential (\emph{cf}. Lemma~\ref{lem: essential}). Note that the fundamental group is a birational invariant, so $\pi_1(M)$ is amenable as well, and a birational morphism $\pi\colon M\to N$ maps the $\R$-fundamental class of $M$ to that of $N$. It follows that $M$ is an essential manifold. Thus, we deduce from~\cite[Theorem 7.2]{Dr11b} that $\dim_{MC}\wt M=\dim_{MC}\wt N=2n$, since the universal cover of an essential manifold with amenable fundamental group has maximum $\dim_{MC}$. 
\end{proof}

\section{Total K\"ahler asphericity and real spherical homology}\label{sec: all aspherical}

In this section, we examine the implications for the topology of the underlying smooth manifold when \emph{all} K\"ahler forms in one of its K\"ahler cones are aspherical. When $M$ is \emph{totally symplectically aspherical} (i.e., all symplectic forms on it are aspherical), then $\Pi(M)$ vanishes (i.e., $M$ cannot have any non-trivial real spherical homology classes). We now study such constraints on the topology of $M$ when $M$ is merely \emph{totally K\"ahler aspherical} (i.e., has a K\"ahler cone in which every K\"ahler form is aspherical). It turns out that the vector space $\Pi(M)$ is again constrained similarly (depending on the complex structure), and in fact, it sometimes vanishes.

\begin{prop}\label{prop: all aspherical}
Let $M$ be a closed K\"ahler manifold with an integrable almost complex structure $J$. If every K\"ahler form on $M$ compatible with $J$ is aspherical, then the dual of $\Pi(M)$ is a subspace of the real Dolbeault group $H_J^{(2,0),(0,2)}(M)_{\R}$, i.e., 
\[
\textup{dual}(\Pi(M))\subset H_J^{(2,0),(0,2)}(M)_{\R}.
\]
\end{prop}

\begin{proof}
    We will utilize the interaction between the cones of $J$-tame and $J$-compatible symplectic forms in this proof. They are defined respectively as
    \[
    K_J^t:=\left\{[\omega]\in H^2(M;\R)\, \mid\,\omega(v,Jv)>0 \text{ for all }v\ne 0\right\}, \ \text{ and},
    \]
    \[
K_J^c:=\left\{[\omega]\in K_J^t\, \mid\,\omega(Ju,Jv)=\omega(u,v) \text{ for all }u,v \right\},
\]
see~\cite{LZ} (and also~\cite{DLZ}). Indeed, $K_J^c$ is open in $H^{1,1}_J(M)_\R$. Suppose that the dual of $\Pi(M)$ is \emph{not} contained in $H_J^{(2,0),(0,2)}(M)_{\R}$. Then there exists a closed $2$-form $\sigma$ on $M$ such that 
    \[
    [\sigma]\in \text{dual}(\Pi(M))\setminus H_J^{(2,0),(0,2)}(M)_{\R}.
    \]
    Since $[\sigma]\ne 0$, we can then use a basis of the finite dimensional vector space
    \[
    H^2(M;\R)\cong H^{1,1}_J(M)_\R\ \oplus\ H_J^{(2,0),(0,2)}(M)_{\R}
    \]
    to write $[\sigma]=[\sigma_1]+[\sigma_2]$, where $[\sigma_2]\in H_J^{(2,0),(0,2)}(M)_{\R}$ and the expression of 
    \[
    0\ne[\sigma_1]\in \text{dual}(\Pi(M))\ \cap\ H^{1,1}_J(M)_\R
    \]
    does not contain any basis elements of $H_J^{(2,0),(0,2)}(M)_{\R}$. Of course,
there exists $x\in \Pi(M)$ such that $\langle[\sigma_1],x\rangle\ne 0$. In particular, $\sigma_1$ is not an aspherical closed $2$-form. If there is \emph{no} aspherical $J$-K\"ahler form on $M$, then there is nothing to prove, and we have already arrived at a contradiction. So, assume that $\omega$ is a compatible aspherical K\"ahler form. For $\lambda>0$ small enough, $\gamma=\omega+\lambda\sigma_1$ is a compatible symplectic form that is not aspherical. Clearly, $[\gamma]=[\omega]+\lambda[\sigma_1]\in K_J^t$. Since $\omega$ is $J$-compatible K\"ahler, we have $[\omega]\in K_J^c$ and so   
    \[
    K_J^t=K_J^c+ H_J^{(2,0),(0,2)}(M)_{\R} \quad \text{and} \quad K_J^c=K_J^t \cap H_J^{1,1}(M)_\R
    \]
    by~\cite{LZ}. Since
    $[\sigma_1]\in H_J^{1,1}(M)_\R$,
    this implies that $[\gamma]\in K_J^c$. Therefore, there is a $J$-compatible K\"ahler form, say $\gamma'$, in the class $[\gamma]$. Note that $\gamma'$ is not aspherical because $\langle[\gamma'],x\rangle = \langle[\gamma],x\rangle \ne 0$.  
%
But this contradicts the hypothesis that every $J$-compatible K\"ahler form on $M$ is aspherical.
\end{proof}

\begin{cor}\label{cor: iff}
Let $(M,J)$ be a closed K\"ahler manifold such that $H^{(2,0),(0,2)}_J(M)_\R$ vanishes. Then $\Pi(M)=0$ if and only if every $J$-compatible K\"ahler form on $M$ is aspherical.   
\end{cor}

\begin{proof}
    The forward direction is obvious (\emph{cf}. Section~\ref{sec: prelim}). For the backward direction, we simply apply Proposition~\ref{prop: all aspherical}.
\end{proof}

The hypothesis of Corollary~\ref{cor: iff} is satisfied if, for instance, $M$ is a K\"ahler surface with $b_2^+(M)=1$, see~\cite{LZ}.

For K\"ahler surfaces, we also get the following result.

\begin{cor}\label{cor: dim 4 all aspherical}
Let $M$ be a closed K\"ahler surface with an integrable almost complex structure $J$. If every $J$-compatible K\"ahler form on $M$ is aspherical, then real spherical homology classes of $M$ cannot be represented by any symplectic $2$-sphere embedded with respect to any symplectic structure on $M$.
\end{cor}

\begin{proof}
Suppose there is a real spherical homology class $0\ne \alpha\in\Pi(M)$ represented by a symplectic $2$-sphere embedded with respect to some symplectic structure on $M$. Let $\beta\in\text{dual}(\Pi(M))$ be its Hom dual. Since every $J$-compatible K\"ahler form on $M$ is aspherical, Proposition~\ref{prop: all aspherical} gives $\beta\in H_J^{(2,0),(0,2)}(M)_\R$. Then the Hodge Index Theorem (see, for instance,~\cite[Chapter~IV]{bpv}) implies that $\beta^2>0$. Thus, it follows from~\cite{mcd} that $M$ is either rational or ruled. But as explained in the proof of Theorem~\ref{th: 2 and 4}, neither of these is possible for a symplectically aspherical $4$-manifold due to fundamental group reasons. Hence, we get a contradiction. 
\end{proof}

\begin{remark}
    It is unclear if the hypothesis in Corollary~\ref{cor: dim 4 all aspherical} can be dropped or relaxed. Indeed, while most spherical homology classes in $M$ can be represented by symplectically \emph{immersed} $2$-spheres in view of the $h$-principle (see~\cite[Section~2]{TJLi1} for a sketch), our proof uses results of~\cite{mcd} which apply to symplectically \emph{embedded} $2$-spheres in $M$. In general, there are several obstructions to being represented by symplectically embedded surfaces, see~\cite[Section~3]{TJLi1}.
\end{remark}

In any case, Corollary~\ref{cor: dim 4 all aspherical} yields the following, which does not involve any assumptions concerning $J$-\emph{homolomorphic} embeddings or immersions in $M$.

\begin{cor}\label{cor: kahler surface}
    If there exists a symplectic form $\omega$ on a closed K\"ahler surface $(M,J)$ and a symplectic embedding $f\colon (S^2,f^*\omega)\to (M,\omega)$, then there exists a $J$-compatible K\"ahler form on $M$ that is not aspherical. 
\end{cor}

\begin{ex}
    Let $M=SP^2(M_g)$ for $g\ge 3$ and $J$ be the complex structure induced on it by a generic complex structure on $M_g$. Let $\omega$ be a symplectic form on $SP^2(M_g)$ obtained using the hyperelliptic complex structure on $M_g$. Then $(SP^2(M_g), \omega)$ has a symplectic sphere (in fact, a rational curve), so Corollary~\ref{cor: kahler surface} implies the existence of a non-aspherical $J$-compatible K\"ahler form on $SP^2(M_g)$. In Section~\ref{sec: non-aspherical}, we will verify this fact by proving the uniqueness of aspherical $J$-compatible K\"ahler forms on $SP^n(M_g)$ for all $n\ge 2$ and $g\ge 2n-1$.
\end{ex}

\section{Uniqueness of aspherical K\"ahler forms on symmetric products}\label{sec: non-aspherical}
In this section, we continue our investigation of K\"ahler cones corresponding to aspherical K\"ahler forms. Our goal now is to describe a family of examples of closed symplectically aspherical smooth projective varieties $(M,J,\omega)$ on each of which there are infinitely many $J$-compatible K\"ahler forms $\omega'$ that are \emph{not} aspherical. In view of Lemma~\ref{lem: exact}, the results of this section show that if $\pi\colon\wt M\to M$ is the universal Riemannian cover and $\pi^*\omega$ is an exact form on $\wt M$, then $\pi^*\omega'$ is not guaranteed to be exact on $\wt M$ for \emph{any} choice of a perturbation $\omega'$ of $\omega$ inside the corresponding $J$-K\"ahler cone on $M$.

Consider a complex curve $M_g$ that is generic in its moduli space $\M_g$, and consider its $n$-th symmetric product $SP^n(M_g)$ for $n\le\lfloor\tfrac{g+1}{2}\rfloor$. 
As explained in~\cite[Theorem~4.1]{DCDJ}, the $n$-th Abel--Jacobi map 
\[
\mu_n\colon SP^n(M_g)\to \C^g/\Lambda =:T^{2g}
\]
is a holomorphic embedding, and so the pullback $\mu_n^*\omega_0$, where $\omega_0$ is the canonical aspherical K\"ahler form on $T^{2g}$, is an aspherical K\"ahler form on $SP^n(M_g)$ (\emph{cf}. Lemma~\ref{lem: sak properties}). We write $\omega:=\mu^*_n\omega_0$, and we use $J$ to denote the corresponding integrable almost complex structure on $SP^n(M_g)$.

\begin{thm}\label{th: unique Kaehler}
Consider the closed symplectically aspherical K\"ahler $2n$-manifold $(SP^n(M_g),J,\omega)$ described above. A $J$-compatible K\"ahler form $\omega'$ on $SP^n(M_g)$ is aspherical if and only if it is cohomologous to a non-zero scalar multiple of $\omega$.
\end{thm}

\begin{proof}
    Our proof uses the structure of the N\'eron--Severi group $N^1(SP^n(M_g))$, so first we recall that. For $p\in M_g$, let $i_{p}\colon SP^{n-1}(M_g)\hookrightarrow SP^n(M_g)$ be the basepoint inclusion $i_p(Q)= p+Q$, where $Q\in SP^{n-1}(M_g)$ is any divisor. The numerical equivalence class of $\text{Im}(i_p)$ is independent of $p$, so we denote it by $x\in N^1(SP^n(M_g))$. By definition, $x$ is the Poincaré dual of $[SP^{n-1}(M_g)]$. Next, let $\Theta$ denote the \emph{theta divisor} on the Jacobian $T^{2g}$, and let $\theta\in N^1(T^{2g})$ be its numerical equivalence class. Let us use the same notation $\theta\in N^1(SP^n(M_g))$ for the numerical equivalence class $\mu_n^*\theta$ in $SP^n(M_g)$. It is well-known (see~\cite{ACGH,Ko}) that the group $N^1(SP^n(M_g))$ is generated by the numerical equivalence classes $x$ and $\theta$.

\begin{figure}[htbp]
    \centering
    \includegraphics[width=0.3\linewidth]{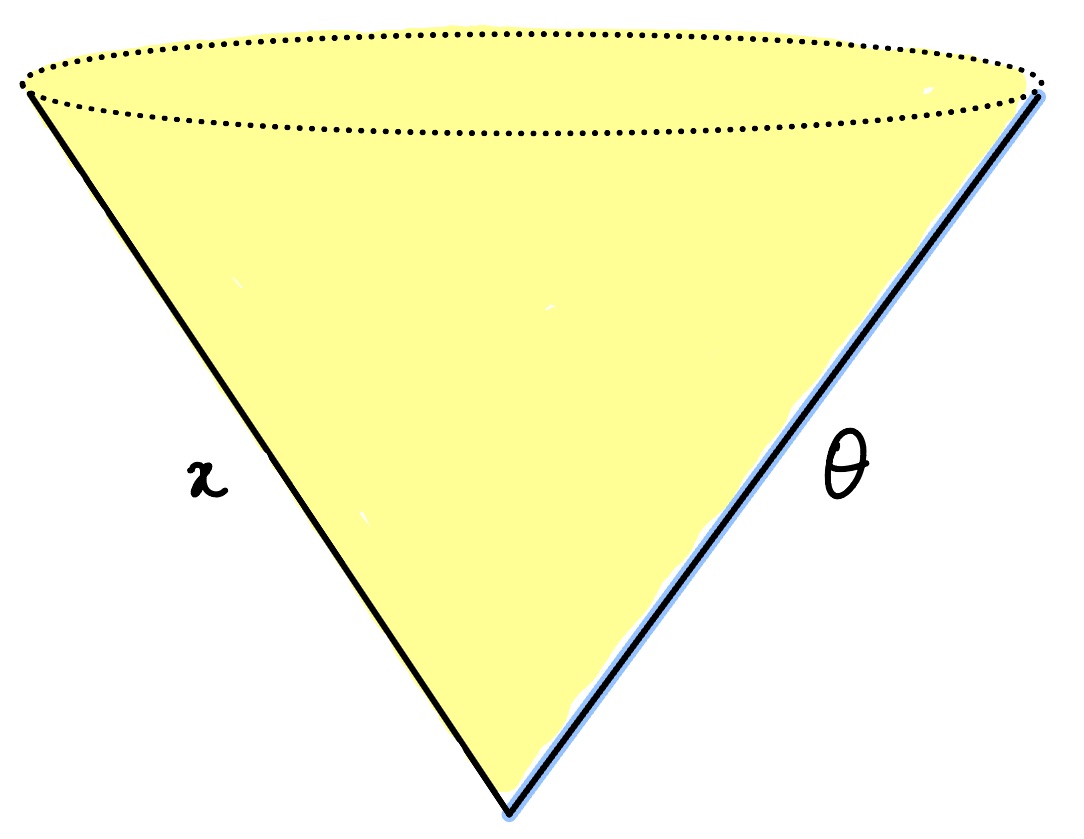}
    \captionsetup{font=footnotesize}
    \caption{ The $J$-K\"ahler cone of $SP^n(M_g)$ spanned by $x$ and $\theta$. The axis of aspherical K\"ahler forms is in blue (on the right), and the portion of non-aspherical K\"ahler forms is in yellow.}
    \label{fig: kahler cone}
\end{figure}

The numerical equivalence class of $\Theta$ in $N^1(T^{2g})$ is the class $\sum_{i=1}^ga_i^*b_i^*$, where $\{a_i,b_i\mid1\le i \le g\}$ is a symplectic basis of $H_1(T^{2g};\R)$, and $*$ in the superscript denotes the respective Hom duals. It follows that 
\[
\theta=\sum_{i=1}^ga_i^*b_i^*\in N^1(SP^n(M_g)),
\]
where we use the same notations $a_i^*$ and $b_i^*$ for the respective images under $\mu_n^*$. Next, take $c=[M_g]\in H_2(SP^n(M_g);\R)$ to be the fundamental class of $M_g$, and its image under $(\mu_n)_*$, and let $c^*$ be its Hom dual. Then we have 
\[
c^*=PD([SP^{n-1}(M_g)])=x. 
\]
We refer the reader to~\cite{ACGH} for more details. Next, recall from~\cite[Sections~3~and~9]{DCDJ} that the real spherical homology class 
\[
u=c-\sum_{i=1}^ga_i\cdot b_i
\]
generates the vector space $\Pi(SP^n(M_g))$. We note that the class of the standard K\"ahler form $\omega_0$ on $T^{2g}$ is $\sum_{i=1}^ga_i^*b_i^*$, so we have $[\omega]=\mu_n^*[\omega_0]=\sum_{i=1}^ga_i^*b_i^*=\theta$ in our notations. That $\omega$ is indeed an aspherical form on $SP^n(M_g)$ is then verified by the fact that $\sum_{i=1}^ga_i^*b_i^*$ evaluates to $0$ on $u$, see~\cite[Proposition~9.7]{DCDJ}. 

Now, let $\omega'$ be an arbitrary $J$-compatible K\"ahler form on $SP^n(M_g)$. By what has been said above, there exist $\alpha,\beta\in\Z$ such that $[\omega']=\alpha c^*+\beta\sum_{i=1}^ga_i^*b_i^*$, and $\omega'$ is aspherical if and only if
\[
0=\langle[\omega'],u\rangle=\alpha\langle c^*,u\rangle+\beta\left\langle\sum_{i=1}^ga_i^*b_i^*,u\right\rangle=\alpha\langle c^*,u\rangle.
\]
To complete our proof, we must show that $\langle c^*,u\rangle\ne 0$, because then, $\omega'$ will be aspherical if and only if $[\omega']=\beta[\omega]=[\beta\omega]$ for the K\"ahler form $\beta\omega$. Letting $PD$ denote the respective Poincar\'e duals, we have that
\[
\langle c^*,u\rangle  = c^*\smile PD(u)=c^*\smile\left(PD(c)-\sum_{i=1}^g PD(a_i\cdot b_i)\right)
\]
\[
=c^*\smile\left((c^*)^{n-1}-\sum_{i=1}^g PD(a_i\cdot b_i)\right)=1-\sum_{i=1}^g\left(c^*\smile PD(a_i\cdot b_i)\right)
\]
by definition. It remains to find the Poincar\'e duals $PD(a_i\cdot b_i)$ for $1\le i \le g$. We claim that $PD(a_i\cdot b_i)=(c^*)^{n-1}-a_i^*b_i^*(c^*)^{n-2}$ for each $i$. If this is the case, then
\[
c^*\smile PD(a_i\cdot b_i)=(c^*)^{n-1}(c^*-a_i^*b_i^*)=0,
\]
where the equalities hold due to Macdonald's relations (see~\cite{Mac} or below) on the algebra $H^*(SP^n(M_g);\R)$, and then this will give the desired $\langle c^*,u\rangle=1$. All that remains now is to prove the claim that $PD(a_i\cdot b_i)=(c^*)^{n-1}-a_i^*b_i^*(c^*)^{n-2}$ for each $i$. To this end, we will make use of Macdonald's relations that
\[
a_{i_1}^*\cdots a_{i_l}^*b_{j_1}^*\cdots b_{j_m}^*(c^*-a_{k_1}^*b_{k_1}^*)\cdots(c^*-a_{k_r}^*b_{k_r}^*)(c^*)^s=0
\]
whenever $l+m+2r+s> n$ for any distinct set of indexes $i_1,\dots,i_l$, $j_1,\dots,j_m$, and $k_1,\dots, k_r$, see~\cite{Mac} (and also~\cite{KS} for additional details).

Fix $i\in\{1,\dots,g\}$. We have from~\cite[Lemma 15]{KS} that $(a_i\cdot b_i)^*=a_i^*b_i^*-c^*$ for the Hom dual of $a_i\cdot b_i$, so it suffices to find a basis of $H^2(SP^n(M_g);\R)$ containing $a_i^*b_i^*-c^*$ such that 
\[
PD(a_i\cdot b_i)(a_i^*b_i^*-c^*)=(c^*)^n=1
\]
and the cup product of $PD(a_i\cdot b_i)$ with any other element of the basis is $0$. We prove that the set of classes $a_k^*a_\ell^*$ and $b_k^*b_\ell^*$ for $k<\ell$, $a_k^*b_{\ell}^*$ for $k\ne \ell$, $a_k^*b_k^*-c^*$ for all $k$, and $c^*$, form such a basis. That this is a basis is standard (see~\cite[Section~3]{DCDJ} and the references therein). Note that because $(a_i^*b_i^*)^2=0$, we have
\[
((c^*)^{n-1}-a_i^*b_i^*(c^*)^{n-2})(a_i^*b_i^*-c^*) = (c^*)^{n-1}a_i^*b_i^* 
\]
\[
= (c^*)^{n} - (c^*)^{n-1}(c^*-a_i^*b_i^*)=(c^*)^n=1.
\]
Next, because $a_j^*b_j^*=-b_j^*a_j^*$ for all $j$, we see for $k < \ell$ that 
\[
((c^*)^{n-1}-a_i^*b_i^*(c^*)^{n-2})(a_k^*a_\ell^*) = (c^*)^{n-1}a_k^*a_\ell^* + (c^*)^{n-2}b_i^*a_i^*a_k^*a_\ell^*. 
\]
The first term above is $0$. If $k=i$, the second is directly $0$. If $k \ne i$, then 
\[
(c^*)^{n-2}b_i^*a_i^*a_k^*a_\ell^* = (c^*)^{n-1}a_k^*a_\ell^*-\left[(c^*)^{n-2}(c^*-b_i^*a_i^*)a_k^*\right]a_\ell^*,
\]
where the first term is $0$ since $k \ne \ell$, and term inside the square bracket is $0$ since $k \ne i$. Similarly, it can be shown that $((c^*)^{n-1}-a_i^*b_i^*(c^*)^{n-2})(b_k^*b_\ell^*)=0$ whenever $k<\ell$. Next, we have for $k\ne i$ that
\[
((c^*)^{n-1}-a_i^*b_i^*(c^*)^{n-2})(a_k^*b_k^*-c^*)= 0 + b_i^*\left[(c^*)^{n-2}a_i^*(a_k^*b_k^*-c^*)\right]=0
\]
because the term in the square bracket is $0$ as before. For any $k\ne \ell$, note that
\[
((c^*)^{n-1}-a_i^*b_i^*(c^*)^{n-2})(a_k^*b_\ell^*)= (c^*)^{n-2}b_i^*a_i^*a_k^*b_\ell^*. 
\]
If $k=i$, the above term is directly $0$. If $k \ne i$, then we again have
\[
(c^*)^{n-2}b_i^*a_i^*a_k^*b_\ell^*=b_i^*[(c^*)^{n-2}a_i^*a_k^*b_{\ell}^*]=0
\]
because $i$, $k$, and $\ell$ are distinct. The equality $((c^*)^{n-1}-a_i^*b_i^*(c^*)^{n-2})c^*=0$ has already been shown. This completes the proof.
\end{proof}

\section*{Acknowledgment}
The authors acknowledge the support of NSF-DMS-2601710 during their travel to the International Conference on Complex and Differential Geometry (ICCDG) at the International Centre for Theoretical Physics (ICTP) in Trieste, Italy, in the period May 25–29, 2026.

LFDC thanks Simone Cecchini for pointing out his work with Francesco Bei and for explaining its relevance to some of the topics addressed in Section~\ref{sec: psc}. He is partially supported by the Simons Foundation grant MPS-TSM-00013711.

AD is grateful to the organizers of the conferences in honor of Sergei Novikov (University of Maryland, May 4-8, 2026) and Claude LeBrun (ICTP, May 25-29, 2026) for the invitation to present this project.

EJ thanks Tedi Draghici for discussions on tame and compatible symplectic cones during his visit to ICTP for the ICCDG. He is grateful to Weiyi Zhang for suggesting corrections to the proof of Proposition~\ref{prop: all aspherical}. He also acknowledges partial travel support from the Clay Mathematics Institute, the Foundation Compositio Mathematica, and the Simons Foundation for his travel to ICCDG.

\end{document}